\documentclass[12pt]{amsart}
\usepackage{txfonts}

\usepackage{amssymb}
\usepackage{amsfonts}
 \usepackage{amsfonts,amssymb}

\usepackage{mathrsfs}
 
\usepackage[all,ps,cmtip]{xy}

\def\llra{\hbox to 10mm{\rightarrowfill}}

\def\lllra{\hbox to 15mm{\rightarrowfill}}

\def\phi{{\varphi}}

\def\cI{\mathcal{I}}

\def\cF{\mathcal{F}}

\def\cO{\mathcal{O}}

\def\cP{\mathcal{P}}

\def\cM{\mathcal{M}}

\def\cW{\mathcal{W}}

\let\tilde\widetilde

\DeclareMathOperator{\rank}{rank}

\DeclareMathOperator{\codim}{codim}
\DeclareMathOperator{\Pic}{Pic}
\DeclareMathOperator{\Hom}{Hom}

\DeclareMathOperator{\Alb}{Alb}

\newtheorem{lemm}{Lemma}[section]
\newtheorem{theo}[lemm]{Theorem}
\newtheorem{coro}[lemm]{Corollary}

\newtheorem{prop}[lemm]{Proposition}

\newtheorem*{conj*}{Conjecture}

\theoremstyle{definition}

\newtheorem{rema}[lemm]{Remark}

\newtheorem{exam}[lemm]{Example}

\theoremstyle{remark}
\newtheorem*{remark*}{Remark}
\newtheorem*{note*}{Note}
 \theoremstyle{remark}

\begin{document}
\title[Surfaces of general type with $p_g=q=2$]{Cohomological rank functions and surfaces of general type with $p_g=q=2$}
\author{Jiabin Du}
\address{Shanghai Center for Mathematical Sciences, China}

\email{jiabindu@fudan.edu.cn}
\author{  Zhi Jiang}
\address{Shanghai Center for Mathematical Sciences, China}
\email{zhijiang@fudan.edu.cn}

\author{ Guoyun Zhang}
\address{Shanghai Center for Mathematical Sciences, China}
\email{gyzhang21@m.fudan.edu.cn}
\date{\today}



\begin{abstract}
 We classify minimal surfaces $S$ with $p_g=q=2$ and $K_S^2=5$ or $6$.
\end{abstract}
\maketitle

\section{Introduction}
We work over complex number field $\mathbb C$ throughout this article.

 A minimal surface $S$ of general type with $\chi(\cO_S)=\chi(\omega_S)=1$ is in some sense the simplest surface of general type since the important birational invariant $\chi(\cO_S)$ takes the minimal value. The classification of such surfaces is a major problem in surface theory (see \cite[Section 2]{BCP}).

 For these surfaces, Beauville pointed out in the appendix of \cite{D} that $p_g=q\leq 4$ and $S$ is isomorphic to the product of two smooth projective curves of genus $2$ when the equality holds.

 When $p_g=q=3$ or  $p_g=q=2$ and the Albanese morphism $a_S$ is a fibration onto a curve, a full classification has been worked out  due to the effort of several people (see \cite{HP, Pg, Z, Pm}).

  The case when $p_g=q=2$ and $a_S$ is generically finite is already wildly open (see \cite{AC, Pm2}). We briefly mention in the following the state-of-art of  this topic.

 By Debarre's inequality and Miyaoka-Yau's inequality, we know that $4\leq K_S^2\leq 9$. When $K_S^2=4$,  we know that $S$ is a double cover of a  principally polarized abelian surface $(A, \Theta)$ branched over a divisor $D\in |2\Theta|$ with negligible singularities (see \cite{CMLP}). Indeed, the surface is on the Severi line and by the work of \cite{BPS1} and \cite{LZ}, it should be a simple double cover of a principally polarized abelian surface.

 Chen and Hacon constructed the first family of minimal surfaces $S$ of general type with $K_S^2>4$, $p_g=q=2$, and $a_S$ generically finite in \cite{CH}. These surfaces are called Chen-Hacon surfaces. A Chen-Hacon surface $S$ satisfies the property that $K_S^2=5$, $p_g=q=2$, and $\deg a_S=3$. They have also been extensively studied in \cite{PP1} and \cite{AC}.

 Since Chen and Hacon's work, many surfaces with $p_g=q=2$ and $6\leq K_S^2\leq 8$ have been constructed (see \cite{AC,CP, Pm, PP2, PP3,PP4, PRR, R}). There are several different ways to construct these surfaces: double covers of an abelian surface branched over a curve with prescribed singularities,
diagonal or mixed quotients of a product of curves by  finite groups, and  triple or quadruple covers of abelian surfaces.

 Nevertheless, a fine classification theory of surfaces with $5\leq K_S^2\leq 9$ and $p_g=q=2$ was still missing. Even an effective bound of the degrees of the Albanese maps of these surfaces was out of reach.

 This paper solves two cases among the previously unknown five cases.

 \begin{theo}\label{5} A minimal surface $S$ of general type with $K_S^2=5$ and $p_g=q=2$ is a Chen-Hacon surface.
 \end{theo}

This theorem answers various questions raised in  \cite{PP1} or \cite{AC} about the possible structures of minimal surfaces  of general type with $K^2=5$ and $p_g=q=2$.

 We also determined the possible structures of  minimal surfaces  of general type with $K^2=6$ and $p_g=q=2$, which answers open questions asked in \cite[Question 3.3]{Pm2} and \cite[Section 6]{PP3}.
.
 \begin{theo}\label{6}Assume that $S$ is a minimal surface of general type  with $K_S^2=6$ and $p_g=q=2$, then $2\leq \deg a_S\leq 4$. Moreover, surfaces with $2\leq \deg a_S\leq 4$ are exactly those constructed by Penegini-Polizzi in \cite{PP2}, Alessandro-Catanese in \cite{AC}, and Penegini-Polizzi in \cite{PP3}.
 \end{theo}

We outline the proof of the main theorems in the following.

In order to prove Theorems \ref{5} and \ref{6}, the essential point is to bound the degree of the Albanese morphism $a_S: S\rightarrow A$ of $S$. The idea is to study possible structures of $a_{S*}\omega_S$, whose rank is equal to the degree of the Albanese morphism since we may assume that $a_S$ is generically finite and surjective.

Our first main tool is the Chen-Jiang decomposition for $a_{S*}\omega_S$. The Chen-Jiang decomposition for pushforward (or higher direct images) of canonical bundles to abelian varieties was first proved in \cite{CJ} and generalized in \cite{PPS}. It is a very powerful tool to deal with the birational geometry of varieties of maximal Albanese dimensions or irregular varieties (see \cite{CJT, CLP, DJL, J}).
We  show in Section 3 that the Chen-Jiang decomposition of $a_{S*}\omega_S$ is simple in the sense that it is the direct sum of $\cO_A$ with a M-regular sheaf $\cM$ and hence $\deg a_S=\rank \cM+1$.

 Our next task is to analyse $\cM$. Chen and Hacon already observed in \cite{CH} that the Fourier-Muaki transform of the derived dual of $\cM$ is the shift of a rank $1$ torsion-free sheaf $L\otimes\cI_Z$ on the dual abelian surface $\hat{A}$ of $A$, where $L$ is an ample line bundle on $\hat{A}$ and $Z$ is a finite subscheme of $\hat{A}$ (see Section 4). It is then natural to consider the base change
   \begin{eqnarray*}
   \xymatrix{
   \hat{S}\ar[r]^{\hat{a}}\ar[d] & \hat{A}\ar[d]^{\phi_L}\\
   S\ar[r]& A,}
   \end{eqnarray*}
   where $\phi_L: \hat{A}\rightarrow  A$ is the standard isogeny induced by $L$.
 With the polarization $L$, we can apply  cohomological rank functions developed in \cite{JP} to $\hat{a}_*\omega_{\hat{S}}$.  The computation of cohomological rank functions provides an important intersection number $(K_{\hat{S}}\cdot \hat{a}^*L)=4h^0(L)^2$ (see Lemma \ref{crf1}).
We then  study in detail the linear system $|K_{\hat{S}}-\hat{a}^*L|$, which leads us to the proof of Theorem \ref{5} and \ref{6}.

The proof of Theorem \ref{5} is a bit easier, since surfaces with $K_S^2=5$ are, in some sense,  on the line of Barja-Pardini-Stoppino equality (see Subsection 2.6) and thus cohomological rank functions provide enough information to show that the  Albanese morphisms are always of degree $3$.

 The proof of Theorem \ref{6} requires much more effort. We first apply Castelnuovo's theory  on genus of space curves and other geometric considerations to show that the degrees of the Albanese morphisms are $\leq 7$.  Then we need to apply our third main tool, namely the symmetry induced by the finite Heisenberg groups. By  considering the symmetry, we  exclude many cases, which a priori could occur when $K_S^2=6$ and finish the proof of Theorem \ref{6}.  This part is inspired by the work of Alessandro-Catanese in \cite{AC}.
\subsection*{Acknowledgements}
The authors thank Fabrizio Catanese and Carlos Rito for comments.
The second author thanks Mart\'i Lahoz for discussions on this topic several years ago in Orsay and thanks Songbo Lin for his interests about possible applications of the Chen-Jiang decomposition on this topic.

 The second author is a member of the Key Laboratory of Mathematics for Nonlinear Science, Fudan University and he is supported  by the Foundation for Innovative Research Groups of the Natural Science Foundation of China (No.~12121001), by the National Key Research and Development Program of China (No.~2020YFA0713200), and by the Natural Science Foundation of Shanghai (No.21ZR1404500, No.23ZR1404000).

 \section{Preliminaries}

 \subsection{The Fourier-Mukai transform}
 Let $X$ be a smooth projective variety.  We denote by $\mathrm D^b(X)$ the derived category of bounded complexes of coherent sheaves.
 For any object $\cF\in \mathrm D^b(X)$, we denote by $\cF^{\vee}:=\mathcal{H}om(\cF, \cO_X)\in \mathrm D^b(X)$ the derived dual object.

 Let $A$ be an abelian variety of dimension $g$ and $\hat{A}=\Pic^0(A)$ be the dual abelian variety. Let $\cP$ be the normalized Poincar\'e bundle on $A\times \hat{A}$ and $p_1$ and $p_2$ be respectively the projection from $A\times \hat{A}$ to $A$ and $\hat{A}$. Then the functor between the derived categories of bounded complexes of coherent sheaves
 \begin{eqnarray*}\Phi_{\cP}: \mathrm D^b(A)\rightarrow \mathrm D^b(\hat{A})\\
 \star\rightarrow p_{2*}(p_1^*(\star)\otimes \cP)
 \end{eqnarray*}
 is called the Fourier-Mukai transform.   The functor with the same kernel     \begin{eqnarray*}\Psi_{\cP}: \mathrm D^b(\hat{A})\rightarrow \mathrm D^b(A)\\
 \star\rightarrow p_{1*}(p_2^*(\star)\otimes \cP),
 \end{eqnarray*}
 is called the inverse Fourier-Mukai transform. It is shown in (\cite{Mu}) that
 \begin{eqnarray}\label{FM1}\Psi_{\cP}\circ \Phi_{\cP}\simeq (-1_A)^* [g] \nonumber\\\Phi_{\cP}\circ \Psi_{\cP}\simeq (-1_{\hat{A}})^* [g]
 \end{eqnarray}
 and hence $\Phi_{\cP}$ and $\Psi_{\cP}$ induce  derived equivalences between the two derived categories.

For any object $\cF\in \mathrm D^b(\hat{A})$, it is shown in (\cite[Formula 3.8]{Mu}) that
\begin{eqnarray}\label{FM2}\Psi_{\cP}(\cF)^{\vee}\simeq (-1_{A})^*\Psi_{\cP}(\cF^{\vee})[g].
\end{eqnarray}

 Let $\cF$ be a coherent sheaf on an abelian variety $A$.  Denote by $$V^i(\cF):=\{P\in \Pic^0(A)\mid H^i(A, \cF\otimes P)\neq 0\}$$ the $i$-th cohomological support loci. We say that $\cF$ is IT$^0$ if $V^i(\cF)=\emptyset$ for each $i>0$.  By cohomology and base change, $\cF$ is IT$^0$ if and only if $\Phi_{\cP}(\cF)=R^0\Phi_{\cP}(\cF)$ is a vector bundle on $\hat{A}$. In this case, we write $\widehat{\cF}:=R^0\Phi_{\cP}(\cF)$.

 It is well-known that if $L$ is an ample line bundle on $A$,  $L$ is IT$^0$. Moreover, if we denote by $\varphi_L: A\rightarrow \hat{A}$ the isogeny induced by $L$, then $\varphi_L^*\widehat{L}\simeq H^0(L)\otimes L^{-1}$.

 We say that $\cF$ is GV (resp. M-regular) if $$\codim_{\Pic^0(A)}V^i(\cF)\geq i$$ (resp. $\codim_{\Pic^0(A)}V^i(\cF)> i$) for each $i>0$.

 It is also known (see \cite{PaPo2} and \cite{PaPo3}) that $\cF$ is GV (resp. M-regular) iff $\Phi_{\cP}((\cF)^{\vee})[g]=R^g\Phi_{\cP}((\cF)^{\vee})$ is a coherent sheaf (resp. torsion-free coherent sheaf). Moreover, Pareschi and Popa showed in \cite{PaPo1} that if $\cF$ is M-regular, it is continuously globally generated, which means that for any non-empty Zariski open subset $U\subset \Pic^0(A)$, the evaluation map $$\bigoplus_{P\in U}H^0(\cF\otimes P)\otimes P^{-1}\rightarrow \cF$$ is surjective.
 \subsection{Generic vanishing theory}

 Given a morphism $f: X\rightarrow A$ from a smooth projective variety to an abelian variety, Hacon proved in \cite{H} that the push-forward of the canonical bundle $f_*\omega_X$ is a GV sheaf. This important result was strengthened in \cite{CJ} when $f$ is generically finite and in \cite{PPS} in general. We then know that there exist quotients between abelian varieties with connected fibers $p_i: A\rightarrow A_i$, M-regular sheaves $\cF_i$ on $A_i$, and torsion line bundles $Q_i\in \Pic^0(A)$ such that
 \begin{eqnarray}\label{CJ}
 f_*\omega_X\simeq \bigoplus_i p_i^*\cF_i\otimes Q_i.
 \end{eqnarray}
This formula is called the Chen-Jiang decomposition of $f_*\omega_X$.

\subsection{Cohomological rank functions}

Let $\cF$ be a coherent sheaf on an abelian variety $A$ and $L$ an ample line bundle on $A$.  Recall that the $i$-th cohomological rank functions in \cite{JP}, which are exactly the continuous rank functions defined in \cite{BPS} when $i=0$: $$h_{\cF, L}^i: t\in \mathbb Q\rightarrow \frac{1}{M^{2\dim A}}h^i(A, \pi_M^*\cF\otimes L^{M^2t}\otimes Q)\in \mathbb Q,$$ where $M$ is a  sufficiently divisible integer, $\pi_M: A\rightarrow A$ is the multiplication-by-$M$ map, and $Q\in \Pic^0(A)$ general. We verify easily that $h_{\cF, L}^i(t)$ is independent of the choice of $M$.
It is also know that $h_{\cF, L}^i$ can be extended to a continuous function from $\mathbb R$ to $\mathbb R$ and in a small right or left neighborhood of a rational point $x\in \mathbb Q$, $h^i_{\cF, L}$ is  a polynomial function (see \cite[Theorem A]{JP}).

Cohomological rank functions of $\cF$ have deep connections with the positivity of $\cF$. We just mentioned that when $\cF$ is IT$^0$, $h_{\cF, L}^0(t)=\chi(\cF\otimes L^{\otimes t})$ for $-1\ll t<0$ (see \cite[Theorem 5.2]{JP}).
\subsection{The eventual maps}
Let $f: X\rightarrow A$ be a morphism from a smooth projective variety to an abelian variety.
Let $H$ be a line bundle on $X$ such that $$h^0(X, H\otimes f^*Q)>0$$ for $Q\in \Pic^0(A)$. Barja, Pardini, and Stoppino defined in \cite{BPS} the eventual map of $H$ with respect to $f$. In \cite[Lemma 2.3]{J}, it is showed that there exists a coherent subsheaf $(f_*H)_c$ of $f_*H$ such that $(f_*H)_c$ is continuously globally generated and $$H^0(A, (f_*H)_c\otimes Q)=H^0(A, f_*H\otimes Q)$$ for $Q\in \Pic^0(A)$ general and the eventual map of $H$ with respect to $f$ is birationally equivalent to the relative evaluation map of $(f_*H)_c$ on $f(X)$. We note that $(f_*\omega_X)_c$ is nothing but the M-regular direct summand of $f_*\omega_X$ in (\ref{CJ}). The eventual map of $K_X$ is called the eventual paracanonical map of $X$.

We can also formally consider the eventual map of $H\langle tL\rangle$ for $t\in \mathbb Q$ as follows. We may take a positive integer $M$ such that $M^2t\in \mathbb Z$ and let
\begin{eqnarray*}
\xymatrix{
X_M\ar[r]^{f_M}\ar[d]^{\mu_M} & A\ar[d]^{\pi_M}\\
X\ar[r]^f & A}
\end{eqnarray*}
be the Cartesian. We say that the eventual map of $H\langle tL\rangle$ with respect to $f$ is   birational (resp. of degree $d$) if the eventual map of $\mu_M^*H\otimes f_M^*(L^{\otimes (M^2t)})$  with respect to $f_M$ is birational (resp. of degree $d$).
This definition is independent of the choice of $M$ (see \cite[Lemma 2.10]{J}) and hence is well-defined.

\subsection{The  Barja-Pardini-Stoppino inequality}

Barja, Pardini, and Stoppino refined the classical Severi inequality in \cite{BPS}. For a minimal surface $S$ of maximal Albanese dimension, they showed that  $K_S^2\geq 5\chi(\omega_S)$ if the Albanese morphism $a_S$ of $S$ is birationally onto its image (see \cite[Corollary 5.6 and Proposition 6.14]{BPS}). Since the surfaces we consider here are covers of abelian surfaces, we need to modify their arguments.
 We can now state a variant of Corollary 5.6 of \cite{BPS}.

 \begin{prop}\label{BPS-inequ}Let $f: S\rightarrow A$ be a   morphism from a minimal surface of general type to an abelian variety such that $f$ is generically finite onto its image.  Let $L$ be an ample line bundle on $A$ such that the eventual map of $K_S\langle tL\rangle$ is birational onto its image whenever $h^0_{f_*\omega_S, L}(t)>0$ for rational $t<0$. Then $$K_S^2\geq 5\chi(\omega_S)$$ and the equality holds iff $$h^0_{f_*\omega_S^{\otimes 2}, L}(2t)=6h^0_{f_*\omega_S, L}(t)$$ for $t\leq 0$.
 \end{prop}
 \begin{proof}

 We can simply mimic the proof of Corollary 5.6 of \cite{BPS}.
Let $F(t)=h^0_{f_*\omega_S^{\otimes 2}, L}(2t)$ and $G(t)=h^0_{f_*\omega_S, L}(t)$. At each rational point $t_0\in \mathbb Q$, the left derivatives $D^-F(t_0)$ and $D^-G(t_0)$ are well-defined.

 The assumption that  the eventual map of $K_S\langle tL\rangle$ is birational onto its image whenever $h^0_{f_*\omega_S, L}(t)>0$ for rational $t<0$ makes sure that $D^-F(t_0)\geq 6D^-G(t_0)$ for each rational number $t_0<0$ (see \cite[Proposition 5.4 (ii)]{BPS}). Since $F$ and $G$ are zero functions for $t$ sufficiently negative, $F(0)\geq 6G(0)$ and equality holds iff $F(t)=6G(t)$ for all $t\leq 0$. Finally, $F(0)=\chi(\omega_S)+K_S^2$ and $G(0)=\chi(\omega_S)$ by generic vanishing.
 \end{proof}

\subsection{Heisenberg groups}
Let $(\hat{A}, L)$ be a polarized abelian surface and we denote by $A:=\Pic^0(\hat{A})$ the dual abelian surface. We assume that $L$ is of type $(\delta_1, \delta_2)$, where $\delta_1|\delta_2$ are positive integers. It is known that $A$ carries a polarization of the same type (see \cite[14.4]{BL}).

We recall a few facts about the kernel group and the finite Heisenberg group of $L$ (see \cite[Section 6]{BL} or \cite[Section 1]{M}).

The kernel of the isogeny \begin{eqnarray*}\varphi_L: & \hat{A}&\rightarrow A \\
&x&\rightarrow t_x^*L\otimes L^{-1},
\end{eqnarray*}
where $t_x$ is the translation by $x$,
is denoted by $K_L$. Note that if $L$ is of polarization type $(\delta_1, \delta_2)$, $$K_L\simeq (\mathbb Z_{\delta_1}\times \mathbb Z_{\delta_2})^2$$
 and the corresponding finite Heisenberg group $\mathcal K_L$ of $L$ is a central extension of $K_L$:
$$1\rightarrow \mu_{\delta_2}\rightarrow \mathcal K_L\rightarrow K_L\rightarrow 0. $$ The finite Heisenberg group can be interpreted as the group of isomorphisms $(\chi, a)$ of $L$: $$\chi: L\xrightarrow{\simeq} t_a^*L, \;\;a\in K_L.$$

The group $\mathcal K_L$ is not commutative. Indeed, there exists a non-degenerate skew-symmetric pairing $e^L: K_L\times K_L\rightarrow \mathbb C^*$ measuring the noncommutativity of $\mathcal K_L$.
Write $A_L$ and $B_L$ two maximally isotropic group of $K_L$ with respect to  the skew-symmetric pairing $e^L$. Then $A_L\simeq \mathbb Z_{\delta_1}\times \mathbb Z_{\delta_2}$ and $B_L\simeq \widehat{A_L}:=\Hom(A_L, \mathbb C^*)\simeq \mathbb Z_{\delta_1}\times \mathbb Z_{\delta_2}$. As a group, $\mathcal K_L$ is isomorphic to $\mu_{\delta_2}\times A_L\times \widehat{A_L}$ with the group law $$(\alpha_1, x_1, l_1)\cdot (\alpha_2, x_2, l_2)=(\alpha_1\alpha_2\cdot l_2(x_1), x_1+x_2, l_1l_2).$$

The representation of $\mathcal K_L$ on the vector space $H^0(\hat{A}, L)$ is isomorphic to  the Schr\"odinger representation, which can be described as the irreducible representation of $\mathcal K_L$ on $$V_L:=\{\mathbb C\textit{-valued functions on } A_L\}$$ via $$(\alpha, x, l)\cdot f(z)=\alpha \cdot l(z)\cdot f(z+x).$$
 \subsection{Examples of surfaces of general type with $p_g=q=2$ and $K_S^2=5,6$ }

 In this subsection, we simply describe the constructions of Chen-Hacon surfaces (CH surfaces for short) constructed in \cite{CH}, Alessandro-Catanese surfaces (AC3 surfaces) constructed in \cite{AC} and \cite{CS} and Penegini-Polizzi surfaces (PP2 and PP4 surfaces) constructed in \cite{PP2, PP3}.

 \begin{exam}[CH surfaces]
 	Given  a  general polarized abelian surface $(\hat{A},L)$ of type $(1,2)$. Assume that $L$ is symmetric.  Then the Fourier-Mukai transform $\cF: =\widehat{L}$ is a vector bundle of rank $2$ on $A=\hat{\hat{A}}$. Let $\phi_{L} : \hat{A}\rightarrow A$ be the the isogeny induced by $L$. By~\cite{M} and \cite{CH} (see also \cite{PP1}), there is a $1$-dimensional family of triple covers $f: X\rightarrow A$ with Tschirnhausen bundle $~\cF$, corresponding to the $2$-dimensional vector space $H^0(A,S^3\cF^{\vee}\otimes\bigwedge^2\cF)$. It is known that a general $X$ has $1$ singular point of type $\frac{1}{3}(1, 1)$.
 Let $\mu:S\to X$   be the resolutions of singularities. Then $S$ is a minimal surface of general type with $K_S^2=5,~ p_g=q=2$, which is called  a Chen-Hacon surface.

 In \cite{PP1}, Penegini and Polizzi extended Chen and Hacon's construction and described in details the moduli space which is connected and irreducible, whose very general point corresponds to Chen and Hacon's construction.
 \end{exam}

 \begin{exam}[PP2 surfaces]
Penegini and Polizzi classified minimal surfaces $S$ of general type with $K_S^2=6$, $p_g(S)=q(S)=2$, and $\deg a_S=2$ in \cite{PP2}.
 For these surfaces, the Albanese variety  is  a polarized abelian surfaces $(A, L)$ of type $(1, 2)$ and the Albanese morphism $a_S: S\rightarrow A$ is branched over a divisor $D\equiv 2L$ whose unique non-negligible singularity is an ordinary quadruple point. Penegini and Polizzi described the moduli space of surfaces constructed in this way, which has $3$ connected irreducible components.
 \end{exam}
 \begin{exam}[PP4 surfaces]
 	The idea of the construction of PP4 surfaces is inspired the construction of  Chen-Hacon surfaces.
 	
 	Start to considering a $(1,3)$-polarized abelian surface $(\hat{A},L)$. Similarly, the Fourier-Mukai transform $\cM:=\widehat{L}^{\vee}$ is a vector bundle of rank 3 and $\phi_{L}^{*}\widehat{L}^{\vee}\cong  H^0(L)^\vee\otimes L$. The idea is to consider quadruple covers of $\hat{A}$ with Tschirnhausen bundle $~\phi_{L}^{*}\cM^\vee$ and identify the covers that descend to quadruple covers  of $ A$ with Tschirnhausen bundle $~\cM^\vee$.
 	
 	Using the data of constructing  quadruple cover  (\cite{HM}), Penegini and Polizzi managed to construct quadruple covers $\hat{a}:\hat{S}\to\hat{A}$ over general $(1,3)$-polarized abelian surface such that $\hat{a}$ descends to a cover $a_S:S\rightarrow A$,  where $S$ is a surface of general type with $K_S^2=6,~p_g=q=2$ and $a_S$ is the Albanese morphism of $S$. The degree of $a_S$ is $4$.
 \end{exam}
 \begin{exam}[AC3 surfaces]
 	Let $(A,L)$ be a polarized abelian surface of type $(1,3)$. We recall notations from the last subsection. Let $\phi_{L}:A\to \hat{A}$ be the isogeny induced by $L$ and its kernel $K_L\cong (\mathbb Z_3)^2$. Denote by $\sigma$ and $\tau$ a generator of $A_L$ and $B_L$ respectively.  Put $V:=H^0(A,L)$, then $V$ is isomorphic to the Schr\"odinger representation of $\mathcal K_L$. We choose a basis $x_0, x_1, x_2$ of
 $V$ such that $\sigma$ acts on $V$ via the cyclic permutation $$x_0\rightarrow x_1\rightarrow x_2\rightarrow x_0$$ and $\tau$ acts on $V$ via $$x_i\rightarrow \xi^ix_i,$$ for $i=0,1,2$, where $\xi$ is a primitive $3$-th root. We also consider the dual representation $V^{\vee}$ and let $y_0, y_1, y_2$ be the corresponding dual vectors. Consider the surfaces $\hat{S}\subset \mathbb P^2\times A$ in the family
 	\begin{equation*}
 		\hat{S}:=\{(y,z)\in \mathbb P^2\times A: \sum_{j=0}^2y_jx_j(z)=0,y_0^2y_1+y_1^2y_2+y_2^2y_0=0\}.
 	\end{equation*}
 	It is clear  that indeed $\hat{S}\subset C\times A$, where $C$ is the cubic curve
 	\begin{equation*}
 		C:=\{y_0^2y_1+y_1^2y_2+y_2^2y_0=0\}\subset \mathbb P^2.
 	\end{equation*}
 Note that $K_L$ acts faithfully on $C$ via the above Schr\"odinger representation and $K_L$ acts on $A$ via the translation. Hence we have a faithful diagonal action of $K_L$ on $\hat{S}$.
 	One verify easily that $C/K_L\simeq \mathbb P^1$ and the quotient surface $S:=\hat{S}/K_L$, called an $AC3$ surface, has $p_g(S)=q(S)=2, K_S^2=6$ and the Albanese morphism of $S$ has degree $3$.
 \end{exam}

 \section{The Chen-Jiang decomposition of $a_{S*}\omega_S$}
 For surfaces $S$ with $p_g=q=2$ and $a_S$ generically finite, we  now consider the Chen-Jiang decomposition of $a_{S*}\omega_S$. Since $a_S$ is generically finite and surjective, we know that $\cO_A$ is a direct summand of $a_{S*}\omega_S$ and since $\chi(a_{S*}\omega_S)=\chi(\omega_S)>0$, the M-regular direct summand of $a_{*}\omega_S$ has holomorphic Euler characteristic $1$. We also know that a torsion-free M-regular sheaf on an elliptic curve is an ample vector bundle. Thus, by the Chen-Jiang decomposition, we have
 $$a_{S*}\omega_S=\cO_A\oplus \cM  \oplus_{i}( p_i^*M_i\otimes Q_i),$$
 where $\cM$ is a M-regular sheaf on $A$ and $\chi(\cM)=1$, $p_i: A\rightarrow E_i$ is a fibration over an elliptic curve, $M_i$ is an ample vector bundle on $E_i$, and $Q_i\in\Pic^0(A)$ is torsion for each $i$.

 We will say that the Chen-Jiang decomposition of $a_{S*}\omega_S$ is simple if $a_{S*}\omega_S=\cO_A\oplus \cM$ is the direct sum of $\cO_A$ with an M-regular sheaf $\cM$.

\begin{theo}Let $S$ be a smooth projective surface of general type such that  $p_g=q=2$ and $a_S$ generically finite. Then the Chen-Jiang decomposition of $a_{S*}\omega_S $ is simple, unless $K_S^2=8$, and $S$ is isogenous to a product of a smooth projective curve of genus $3$ and a smooth projective curve of genus $3$ or $4$ or $5$.
\end{theo}
\begin{proof}

Assume that the Chen-Jiang decomposition of $a_{S*}\omega_S$ contains a direct summand which is of the form $p^*M\otimes Q$, where $p: A\rightarrow E$ is a surjective morphism to an elliptic curve $E$ with connected fibers and $M$ is an ample vector bundle on $E$, and $Q\in \Pic^0(A)$. In particular, $\deg a_S:=\kappaup >2$.

We first remark that $Q\in \Pic^0(A)\setminus p^*\Pic^0(E)$. Indeed, if $Q=p^*Q'\in p^*\Pic^0(E)$, $h^0(A,p^*M\otimes Q )\geq 1$ and then $$2=p_g(S)=h^0(A, a_{S*}\omega_S)\geq h^0(A, \cO_A)+h^0(A, \cM)+h^0(A,p^*M\otimes Q)\geq 3,$$ which is a contradiction.

Note that $Q$ is a torsion line bundle and we denote by $k$ the order of the image of $Q$ in the quotient abelian variety $ \Pic^0(A)/p^*\Pic^0(E)$. We claim that $M$ is a line bundle and moreover there exist ample line bundles $M_i$ on $E$ where $M_1=M$ such that $$\cO_A\oplus \bigoplus_{i=1}^{k-1}(p^*M_i\otimes Q^{\otimes i})$$ is a direct summand of $a_{S*}\omega_S$. This is essentially a special case of \cite[the Claim in Page 2991]{JLT}. We simply state it using the language of Chen-Jiang decomposition.

After replacing $Q$ by a line bundle of order $k$ in $Q\otimes p^*\Pic^0(E)$, we may simply assume that $Q$ is of order $k$ and denote by $\pi: \tilde{S}\rightarrow S$ the cyclic \'etale cover induced by $Q$. By considering the Stein factorization of $p\circ a_S\circ \pi$ and by  \cite[the Claim in Page 2991]{JLT}, we have the following commutative diagram:
\begin{eqnarray*}
\xymatrix{
\tilde{S}\ar[d]^h\ar[r]^{\pi} & S\ar[d]^{p\circ a_S}\\
C\ar[r]^q & E,}
\end{eqnarray*}
where $q: C\rightarrow E$ is a cyclic $\mathbb Z_k$-cover such that $$q_*\omega_C=q_*R^1h_*\omega_{\tilde{S}}=R^1(p\circ a_S)_*\pi_*\omega_{\tilde{S}}=\cO_E\oplus\bigoplus_{i=1}^{k-1}M_i.$$

Note that $g(C)=1+\sum_{i=1}^{k-1}\deg M_i\geq k$. On the other hand, $$\chi(h_*\omega_{\tilde{S}})=\chi(\omega_{\tilde{S}})+\chi(R^1h_*\omega_{\tilde{S}})=k\chi(\omega_S)+\chi(\omega_C)=k+(g(C)-1).$$
We denote by $r$ and $d$ respectively the rank and the degree of $h_*\omega_{\tilde{S}/C}$. Note that $r=g(F)$, where $F$ is a general fiber of $h$ or $p\circ a_S$. We also know that $h_*\omega_{\tilde{S}/C}$ is nef and hence $d\geq 0$ (see \cite{V}).
By Riemann-Roch, we then have
\begin{eqnarray*}
\chi(h_*\omega_{\tilde{S}})=r(g(C)-1)+d=k+g(C)-1.
\end{eqnarray*}
Hence $d+(r-1)(g(C)-1)=k\leq g(C)$. Thus we have the following possibilities
\begin{itemize}\item $d=1$, $r=2$, and $g(C)=k$;
\item $d=0$, $r=2$ and $g(C)=k+1$;
\item $d=0$, $r=3$, and $g(C)=k=2$.
\end{itemize}

 We first remark that if $d=0$, by \cite[Theorem 17.3]{BHPV}, $h$ is locally trivial and hence $h$ is smooth and all fibers of $h$ are isomorphic to $F$. In this case, $K_S^2=8$. It is a bit more tricky to deal with the case that $d=1$.

If $d=1$, either $h$ is isotrivial but not locally trivial or $h$ has non-constant moduli. If $h$ has non-constant moduli, since $g(F)=r=2$, we apply  results of Gang Xiao in \cite{X}. Let $E'$ be the kernel of $p$. Then $E'$ can be seen as the fixed part of the  Jacobian of fibers of $h$. Moreover, the degree of $F$ over $E'$  equal to   $\deg a_{S}=\kappaup$ is $>2$. Thus we can say that $h: \tilde{S}\rightarrow C$ is a fibration of genus $2$ curve of type $(E', \kappaup)$ in the terminology of Xiao (\cite[Chapter 3]{X}). Thus by \cite[Th\'eor\`eme 3.10 and Corollaire in Page 47]{X}, there exists a semistable family $f: S(E', \kappaup)\rightarrow X(\kappaup)$ from a minimal surface to a smooth projective curve and a surjective morphism $\varphi: C\rightarrow X(\kappaup)$ such that $h$ is the desingularization of the pull-back of $f$ and moreover both $f$ and $h$ are semistable. By \cite[Lemma 3.1]{V}, $$h_*\omega_{\tilde{S}/C}=\varphi^*f_*\omega_{S(E', \kappaup)/X(\kappaup)}.$$ The numerical invariants of $S(E', \kappaup)$ and $X(\kappaup)$ were also computed in \cite[Th\'eor\`eme 3.13]{X}. Since $$\deg f_*\omega_{S(E', \kappaup)/X(\kappaup)}=\chi(\omega_{S(E', \kappaup)})-g(X(\kappaup))+1,$$ we see that $\deg f_*\omega_{S(E', \kappaup)/X(\kappaup)}$ is $\geq 2$ when $\kappaup\geq 4$ and is equal to $1$ when $\kappaup=3$ (see \cite[Page 48 and 52]{X}). When $\kappaup=3$, $X(3)=\mathbb P^1$. Thus $\deg h_*\omega_{\tilde{S}/C}\geq 2$ in any case and we get a contradiction.

When $h$ is locally trivial or isotrivial and $g(F)=r=2$, $S$ is birational to a diagonal quotient $(F\times C')/G$, where $G$ acts faithfully on both $F$ and $C'$ and both $F/G$ and $C'/G$ are elliptic curves since $q(S)=2$ and $a_S$ is generically finite. It is elementary that a Galois cover from a genus $2$ curve to an elliptic curve has to be a double cover. Thus $\deg a_S=2$ and we have a contradiction.

The only rest of case is that $d=0$, $r=3$ and $g(C)=k=2$. The fibration $h$ is locally trivial and thus $K_{S}^2=8$. Thus $\tilde{S}$ is isogenous to $F$ with a cover $C'$ of $C$. We can check Penegini's list \cite[Theorem 1.1]{Pm} that $g(C')=3$, or $4$, or $5$.
 \end{proof}

 \begin{exam}
 Let $a_i: C_i\rightarrow E_i$ be two $(\mathbb Z_2\times \mathbb Z_2)$-cover of elliptic curves such that $a_{i*}\omega_{C_i}=\cO_{E_i}\oplus L_i\oplus P_i\oplus (L_i\otimes P_i)$  for $i=1, 2$, where $L_i$ and $P_i$ are respectively degree $1$ and torsion line bundles. Note that $a_i$ is nothing but a composition of a ramified double cover with an \'etale double cover.

 We may also assume that $\mathrm{Gal}(C_1/E_1)=\mathrm{Gal}(C_2/E_2)=G=\langle \sigma\rangle \times \langle \tau\rangle$ and that  $L_1$ (resp. $P_2$) is the $\sigma$-invariant in $a_{1*}\omega_{C_1}$ (resp. $a_{2*}\omega_{C_2}$ )and $P_1$ (resp. $L_2$) is the $\tau$-invariant part in $a_{1*}\omega_{C_1}$ (resp. $a_{2*}\omega_{C_2})$. Let $S$ be the diagonal quotient $(C_1\times C_2)/G$ and $a: S\rightarrow E_1\times E_2$ be the natural morphism (which is also the Albanese morphism of $S$), we have $$a_{S*}\omega_S=\cO_{E_1\times E_2}\oplus (L_1\boxtimes P_2) \oplus (P_1\boxtimes L_2) \oplus ((L_1\otimes P_1)\boxtimes (L_2\otimes P_2)).$$
 \end{exam}

\section{Cohomological rank functions of surfaces of general type with $p_g=q=2$}

We assume in this section that $a_{S*}\omega_S=\cO_A\oplus \cM$ and $\cM$ is M-regular. Chen and Hacon studied the Fourier-Mukai transform of $\cM$ in \cite[Section 3]{CH}. The following paragraph is essentially due to them.

Since $\cM$ is M-regular, we know that $\Phi_{\cP}(\cM^{\vee})[2]=L\otimes \cI_Z$ is a rank $1$ torsion-free sheaf, where $L$ is a line bundle on $\Pic^0(A)$ and $\cI_Z$ is the ideal sheaf of a finite subscheme $Z$ of $\Pic^0(A)$. We then consider the short exact sequence $$0\rightarrow L\otimes \cI_Z\rightarrow L\rightarrow L|_Z\rightarrow 0$$ on $\Pic^0(A)$. We have $$\mathcal{E}xt^2(L|_Z, \cO_{\hat{A}})\simeq \mathcal{E}xt^1(L\otimes\cI_Z, \cO_{\hat{A}})\simeq R^1\Phi_{\cP}(\cM),$$ where we get the second isomorphism by  \cite[Corollary 2.6]{PaPo3}. Thus, it is easy to see that $\mathrm{Supp}(Z)=\mathrm{Supp} (R^1\Phi_{\cP}(\cM))= V^1(\cM)= V^1(a_{S*}\omega_S)\setminus\{\cO_A\}$ \footnote{Here by abuse of notation, we identify a numerically trivial line bundle on $A$   to the corresponding point of $\hat{A}$.}, where the second equality holds because $V^2(\cM)=\emptyset$.
Moreover, by (\ref{FM1}) and (\ref{FM2}), we have $$\Psi_{\cP}((L\otimes \cI_Z)^{\vee})[2]=(-1_A)^*\Psi_{\cP}(L\otimes \cI_Z)^{\vee}=\cM.$$ It is also clear that $$\Psi_{\cP}((L|_Z)^{\vee})[2]=R^2\Psi_{\cP}((L|_Z)^{\vee})$$ is a vector bundle, which is a successive extension of numerically trivial line bundles $P$ which belongs to $V^1(\cM)$. Hence  $$\Psi_{\cP}(L^{\vee})[2]=R^2\Psi_{\cP}(L^{\vee})$$ and $L$ is an ample line bundle.  As before, $$R^2\Psi_{\cP}((L|_Z)^{\vee})=(-1_A)^*\widehat{L|_Z}^{\vee}$$ and $$R^2\Psi_{\cP}(L^{\vee})=(-1_A)^*\widehat{L}^{\vee}$$
We then get a new exact sequence
 \begin{eqnarray}\label{exact}
0\rightarrow(-1_A)^*\widehat{L|_Z}^{\vee}\rightarrow (-1_A)^* \widehat{L}^{\vee}\rightarrow \cM\rightarrow 0
\end{eqnarray}
 on $A$, after taking the inverse Fourier-Mukai transform and taking duality.

We observe that $\rank \cM=\deg a_S-1\leq h^0(L)$ and equality holds iff $Z=\emptyset$.

It turns out that it is easier to work on $\Pic^0(A)=\hat{A}$. Let $\phi_L: \hat{A}\rightarrow \hat{\hat{A}}=A$ be the isogeny induced by $L$ and let $\hat{S}:=S\times_{A}\hat{A}$. We consider the Cartesian:
\begin{eqnarray*}
\xymatrix{
\hat{S}\ar[d]\ar[r]^{\hat{a}}& \hat{A}\ar[d]^{\phi_L}\\
S\ar[r]^{a_S} & A.}
\end{eqnarray*}
It is clear that $\hat{a}_*\omega_{\hat{S}}=\cO_{\hat{A}}\oplus \phi_L^*\cM$.
Since $\phi_L^*\widehat{L}\simeq H^0(L)\otimes L^{-1}$, $$\phi_L^*((-1_A)^*\widehat{L}^{\vee})\simeq H^0(L)^{\vee}\otimes (-1_{\hat{A}})^*L.$$ Thus we have a short exact sequence on $\hat{A}$:
\begin{eqnarray}\label{exact1}0\rightarrow \phi_L^*((-1_A)^*\widehat{L|_Z}^{\vee})\rightarrow H^0(L)^{\vee}\otimes (-1_{\hat{A}})^*L\rightarrow \phi_L^*\cM\rightarrow 0.\end{eqnarray}
We finally remark that $\phi_L^*((-1_A)^*\widehat{L|_Z}^{\vee})$ is again a successive extension of numerically trivial line bundles on $\hat{A}$ and $(-1_{\hat{A}})^*L$ is algebraically equivalent to $L$.
 Thus we are now able to compute the cohomological rank function of $\hat{a}_{*}\omega_{\hat{S}}$.

\begin{lemm}\label{crf} For $-1\leq t\leq 0$, we have
\begin{eqnarray*}&&h^0_{\hat{a}_{*}\omega_{\hat{S}}, L}(t)=h^0(L)^2(t+1)^2\\
&&h^1_{\hat{a}_{*}\omega_{\hat{S}}, L}(t)=(h^0(L)+1-\deg a_S)h^0(L)t^2\\
&&h^2_{\hat{a}_{*}\omega_{\hat{S}}, L}(t)=h^0(L)t^2.
\end{eqnarray*}
\end{lemm}

\begin{proof}
We denote by $\cW=\phi_L^*((-1_A)^*\widehat{L|_Z}^{\vee})$, whose rank is equal to the length of $Z$ and is also equal to $h^0(L)+1-\deg a_S$.

It suffices to note that when $-1\ll t<0$,
\begin{eqnarray*}
&&h^i_{\cW, L}(t)\\&=&(h^0(L)+1-\deg a_S)h^i_{\cO_{\hat{A}}, L}(t)=\begin{cases}
0, &\text{if $i=0$\; or\; $1$},\\
(h^0(L)+1-\deg a_S)h^0(L)t^2, &\text{if $i=2$},
\end{cases}
\end{eqnarray*}
and
\begin{equation*}
h^i_{(-1_{\hat{A}})^*L, L}(t)=\begin{cases}
h^0(L)(1+t)^2, &\text{if $i=0$},\\
0, &\text{if $i=1$\; or\; $2$}.
\end{cases}
\end{equation*}
\end{proof}
It is also clear that the Hilbert-Poincar\'e polynomial $$\chi(\hat{S}, K_{\hat{S}}+t\hat{a}^*L)=\chi(\omega_{\hat{S}})+\frac{1}{2}(K_{\hat{S}}\cdot \hat{a}^*L)t+\frac{1}{2}(\hat{a}^*L)^2t^2$$ is nothing but $$h^0_{\hat{a}_{*}\omega_{\hat{S}}, L}(t)-h^1_{\hat{a}_{*}\omega_{\hat{S}}, L}(t)+h^2_{\hat{a}_{*}\omega_{\hat{S}}, L}(t).$$ Comparing the coefficient of $t$, we get an equality.
\begin{lemm}\label{crf1} $(K_{\hat{S}}\cdot \hat{a}^*L)=4h^0(L)^2$.
\end{lemm}

We believe that Lemma \ref{crf1} is crucial in the classification of all surfaces with $p_g=q=2$.

\begin{coro}Let $S$ be a minimal surface with $p_g=q=2$ and $K_S^2=9$.  Then $\mathrm{length}(Z)\geq 2$ and in particular $V^1(\cM)\neq \emptyset$. Similarly, if $S$ be a minimal surface with $p_g=q=2$ and $K_S^2=8$,  then $V^1(\cM)\neq \emptyset$.
\end{coro}

\begin{proof}  By Hodge index, we have $$(K_{\hat{S}}\cdot \hat{a}^*L)\geq \sqrt{K_{\hat{S}}^2\cdot (\hat{a}^*L)^2}.$$ If $K_S^2=9$, $4h^0(L)^2\geq \sqrt{9h^0(L)^2\cdot( 2(\deg\hat{a})h^0(L))}$. Since $\deg\hat{a}=1+h^0(L)-\mathrm{length}(Z)$, we have $\mathrm{length}(Z)\geq 1+\frac{h^0(L)}{8}$.

If $K_S^2=8$, we may assume that the Chen-Jiang decomposition of $a_{S*}\omega_S$ is simple, then we have $(K_{\hat{S}}\cdot \hat{a}^*L)\geq \sqrt{K_{\hat{S}}^2\cdot (\hat{a}^*L)^2}.$ Similarly, we have $\mathrm{length}(Z)\geq 1$.
\end{proof}

\begin{lemm}\label{2-crf}For $-1\ll t\leq 0$, $$h^0_{\hat{a}_{*}\omega_{\hat{S}}^{\otimes 2}, L}(2t)=h^0(L)^2(1+K_S^2)+3(K_{\hat{S}}\cdot \hat{a}^*L)t+4(\deg a_S)h^0(L)t^2.$$
\end{lemm}
\begin{proof}
It is known that $\hat{a_*}\omega_{\hat{S}}^{\otimes 2}$ is IT$^0$. Thus, by \cite[Theorem 5.2]{JP}, $h^0_{\hat{a}_{*}\omega_{\hat{S}}^{\otimes 2}, L}(2t)$ is simply the  Hilbert-Poincar\'e polynomial $$\chi(\hat{S}, 2K_{\hat{S}}+2t\hat{a}^*L)=h^0(L)^2(1+K_S^2)+3(K_{\hat{S}}\cdot \hat{a}^*L)t+4(\deg a_S)h^0(L)t^2$$ for  $-1\ll t\leq 0$.
\end{proof}

\section{The eventual paracanonical maps}
We still assume that the Chen-Jiang decomposition of $a_{S*}\omega_S$ is simple in this section and then study the eventual paracanonical maps of $S$ or $\hat{S}$.

\begin{lemm}\label{birational}Assume that $\deg a_S\geq 3$, then the eventual map of $K_{\hat{S}}\langle tL\rangle$ for $-1<t\leq 0$ is birational onto its image.
\end{lemm}
\begin{proof} Indeed, by (\ref{exact}), for $-1<t\leq 0$, the eventual map of $K_{\hat{S}}\langle tL\rangle$ is birationally equivalent to the relative evaluation map of (the \'etale pull-back of) $\phi_L^*\cM$. Let $\Phi$ be the relative evaluation map and $Z$ be its image. We have the following commutative diagram:
\begin{eqnarray*}
\xymatrix{\hat{S}\ar[dr]\ar@{.>}[r]^(.3){\Phi} & Z\subset \mathbb P_{\hat{A}}(\phi_L^*\cM)\ar[d] \ar@{^{(}->}[r] &\mathbb P_{\hat{A}}(H^0(L)^{\vee}\otimes (-1_{\hat{A}})^*L)\simeq \hat{A}\times \mathbb P^{h^0(L)-1}\ar[dl]\\
& \hat{A}.}
\end{eqnarray*}
By construction, $\deg(Z/\hat{A})\geq \rank(\cM)=\deg a_S-1$. Thus if $\deg a_S\geq 3$, we have $\Phi: \hat{S}\rightarrow Z$ is birational since $\deg a_S=\deg(\hat{S}/\hat{A})=\deg(\hat{S}/Z)\deg(Z/\hat{A})$.
\end{proof}

We also need to study the eventual paracanonical map more carefully. Let $\phi: a_S^*\cM\rightarrow \omega_S$ be the evaluation map and let $\omega_S\otimes \cI_W$ be the image, where $W$ is a subscheme of $S$.

\begin{lemm}\label{indeterminacy}$W$ is supported on the exceptional locus of $a_S$.
\end{lemm}
\begin{proof}

Let $U\subset A$ be the finite locus of $a_S$. Then by \cite[Theorem 2.1]{CE}, $a_S^*\cM\rightarrow \omega_S$ is surjective over $a_S^{-1}(U)$.
 \end{proof}

 \begin{rema} Since $a_{S*}\omega_S=\cO_A\oplus \cM$ and $\cM$ is M-regular and hence continuously globally generated, we conclude by Lemma \ref{indeterminacy} that,  for any open subset $V$ of $\hat{A}$, $$\bigcap_{P\in V}\mathrm{Bs}(|K_S+P|)$$ is supported on the exceptional locus of $a_S$.
 \end{rema}
Let $\rho: S'\rightarrow S$ be a resolution of indeterminacy of the eventual paracanonical map of $S$ (or equivalently the principalization of $\cI_W$) and let $\Phi': S'\rightarrow \mathbb P_{A}(\cM)$ be the corresponding morphism. We have the following commutative diagram:
\begin{eqnarray}\label{eventual-morphism}
\xymatrix{
& \hat{S}\ar[rr]^{\hat{a}}\ar[dd] && \hat{A}\ar[dd]^(.2){\phi_{L}}\\
\hat{S}'\ar[ur]^{\hat{\rho}}\ar[rr]^{\hat{\Phi}'}\ar[dd] && \mathbb P_{\hat{A}}(\varphi_L^*\cM)\ar[dd]\ar[ur]^{\hat{\pi}}\ar@{^{(}->}[rr]&& \hat{A}\times \mathbb P^{h^0(L)-1}\\
& S\ar[rr]^(.3){a_S}&& A\\
S'\ar[ur]^{\rho}\ar[rr]^{\Phi'} &&\mathbb P_{A}(\cM)\ar[ur]^{\pi},}
\end{eqnarray}
where the upper diagram is obtained via the \'etale base change of $\varphi_L$. We then see that
\begin{eqnarray}\label{canonical-bundle}\Phi'^*\cO_{\pi}(1)=\rho^*(K_S-F_1)-F_2,\end{eqnarray} where $F_1$ is $a_S$-exceptional and $F_2$ is $\rho$-exceptional. We note that $\cO_{\hat{\pi}}(1)\simeq L\boxtimes \cO_{\mathbb P^{h^0(L)-1}}(1)$ and thus we may write
\begin{eqnarray}\label{canonical-bundle1}\hat{\rho}^*(K_{\hat{S}}-\hat{F_1})-\hat{F_2}=\hat{\Phi}'^*( (-1_{\hat{A}})^*L\boxtimes \cO_{\mathbb P^{h^0(L)-1}}(1)),
\end{eqnarray}
where $\hat{F_1}$ and $\hat{F_2}$ are the pull-back of $F_1$ and $F_2$ on $\hat{S}'$.
We will denote by $L^-:=(-1_{\hat{A}})^*L$.

\section{Characterizations of minimal surfaces with $p_g=q=2$ and $K_S^2=5$}
\subsection{Surfaces with $p_g=q=2$ and  $\deg a_S=2$}
Suppose that the Albanese morphism $a_S:S\to A:=\Alb(S)$  is of degree $2$. Consider the Stein factorization of $a_S$
\begin{displaymath}
	\xymatrix{
		S\ar[r]^{p}\ar[rd]&\overline{S}\ar[d]^{f}\\
		&A
	}
\end{displaymath}
then $f: \overline{S}\to A$ is a finite double cover. It is determined by a reduced even divisor $B$ on $A$ and its square root $\delta\in \Pic(A)$ with $\mathcal{O}_{A}(B)\cong\delta^2$. Note that the normality of $S$ implies the  normality of $\overline{S}$. Thus $\overline{S}$ has at most isolated singularities.

After a minimal even resolution $\rho: \widetilde{A}\to A$ of singularities of $B$, we obtain a double cover data $(\widetilde{B},\widetilde{\delta})$ on $\widetilde{A}$ with its image under $\rho$ is $(B,\delta)$, where $\widetilde{B}$ is a smooth reduced even divisor on $\widetilde{A}$ with the square root $\widetilde{\delta}\in \Pic(\widetilde{A})$. More precisely, let $\rho=\rho_1\circ\cdots\circ\rho_{n}$ be a composition of a series of blow-ups $\rho_{i}$, and let $m_i$ be the order of even and reduced  inverse image of $B$ at the center of  $\rho_{i}$. Define $k_i=\lfloor\frac{1}{2}m_i\rfloor$, then one has
\begin{equation*}
	\widetilde{\delta}\sim \rho^{\ast}\delta-\sum_{i=1}^{n}k_iE_i, \widetilde{B}\sim \rho^{\ast}B-2\sum_{i=1}^{n}k_iE_i.
\end{equation*}
where each $E_i$ is the total inverse image of the center of $\rho_{i}$ in $\widetilde{A}$.
Let $\widetilde{f}:\widetilde{S}\to \widetilde{A}$ be the double cover associated to $(\widetilde{B},\widetilde{\delta})$. Then $\widetilde{S}$ is smooth but may be not minimal. Note that
\begin{equation*}
	K_{\widetilde{S}}= \widetilde{f}^{\ast}(\rho^{\ast}(K_A+\delta)-\sum_{i=1}^{n}(k_i-1)E_i).
\end{equation*}
The numerical invariants of $\widetilde{S}$ are
\begin{eqnarray*}
&&	K_{\widetilde{S}}^2=2 (K_A+\delta)^2-2\sum_{i=1}^{n}(k_i-1)^2,
	\\&& \chi(\mathcal{O}_{\widetilde{S}})=2\chi(\mathcal{O}_{A})+\frac{1}{2}(\delta^2+K_A\cdot \delta)-\frac{1}{2}\sum_{i=1}^{n}k_i(k_i-1).
\end{eqnarray*}
 Contracting the $(-1)$-curves on $\widetilde{S}$ yields a birational morphism $\epsilon:\widetilde{S}\to S$.
 It is well-known that $\epsilon$ is a composition of contractions of $(-1)$-curves. We denote by $l$ the numbers of contractions of $(-1)$-curves of $\epsilon$.
  Then
\begin{equation*}
	K_{S}^2=K_{\widetilde{S}}^2+l,\; \chi(\mathcal{O}_{\widetilde{S}})=\chi(\mathcal{O}_{S}).
\end{equation*}
 \begin{lemm}
	Let $S$ be a minimal surface of general type with $p_g=q=2$ and $\deg a_S=2$. Then $K_S^2\neq 5$.
\end{lemm}
\begin{proof}
	By above discussion, one has
	\begin{equation*}
		\chi(\mathcal{O}_{\widetilde{S}})=\frac{1}{2}\delta^2-\frac{1}{2}\sum_{i=1}^{n}k_i(k_i-1),  K_{\widetilde{S}}^2=2 \delta^2-2\sum_{i=1}^{n}(k_i-1)^2.
	\end{equation*}
	Hence
	\begin{equation}
		K_S^2\geq K_{\widetilde{S}}^2=4+2\sum_{i=1}^{n}(k_i-1).
	\end{equation}
	Now suppose that $K_S^2=5$. It follows that  all $k_i=1$ and $K_{\widetilde{S}}=\widetilde{f}^{*}\rho^*\delta$ and thus all the blow-ups centers of $\rho$ are negligible singularities and $\tilde{S}\simeq S$ is minimal. In this case, $K_S^2=4$, a contradiction.
\end{proof}

\begin{theo}\label{degree3}Let $S$ be a minimal surface of general type with $p_g(S)=q(S)=2$ and $K_S^2=5$, the Albanese map $a_S$ is generically finite of degree $3$.
\end{theo}
\begin{proof}
We have already seen that $\deg a_S\geq 3$ when $K_S^2=5$. By Lemma \ref{birational}, we can apply Proposition \ref{BPS-inequ}. Since $K_S^2=5\chi(\omega_S)$, we have   $h^0_{\hat{a}_{*}\omega_{\hat{S}}^{\otimes 2}, L}(2t)=6h^0_{\hat{a}_{*}\omega_{\hat{S}}, L}(t)$ for $t<0$. By Lemma \ref{crf} and Lemma \ref{2-crf}, comparing the coefficient of $t^2$, we have $4\deg a_S=6h^0(L)$. On the other hand, $\deg a_S-1=\rank \cM\leq h^0(L)$. Thus $h^0(L)\leq 2$. We then have $h^0(L)=2$, $\deg a_S=3$, and $\cM=\hat{L}^{\vee}$.
\end{proof}

\begin{coro}\label{ch-surface} A minimal surface $S$ of general type with $p_g(S)=q(S)=2$ and $K_S^2=5$ is a Chen-Hacon surface.
\end{coro}
\begin{proof}
 Let $ a_S: S\rightarrow \overline{S}\xrightarrow{\overline{a}} A$ be the Stein factorization of $a_S$.
By Theorem \ref{degree3}, $a_{S*}\cO_S=\overline{a}_*\cO_{\overline{S}}=\cO_A\oplus \widehat{L}$ is a rank $3$ vector bundle and $K_S^2=2c_1(\widehat{L})^2-3c_2(\widehat{L})=5$. Thus, $\overline{S}$ has negligible singularities by \cite[Definition 1.5]{PP1}. Thus  by  Theorem A  of \cite{PP1}, $S$ is a Chen-Hacon surface. We remark that Corollary \ref{ch-surface} also follows from \cite[Theorem 4.1]{AC}.

\end{proof}

\begin{rema}By the results in this section, we see that the method to prove the Barja-Pardini-Stoppino inequality in \cite{BPS} is indeed sharp. By \cite[Proposition 2.12]{J1}, the Barja-Pardini-Stoppino inequality is not sharp for those surfaces whose Albanese morphism is birational onto its image. Thus
 it is interesting to ask whether or not there exists a sharp Severi type inequality when the Albanese morphism is birational onto its image.
\end{rema}
 \section{Minimal surfaces with $p_g=q=2$ and $K_S^2=6$}

The classification of surfaces $S$ with $p_g=q=2$ and $K_S^2=6$ is more complicated. Note that when $\deg a_S=2$, Penegini and Polizzi had classified the structure of $S$ in \cite{PP2}. Thus we will always assume that $\deg a_S\geq 3$ in this section. In particular, the eventual paracanonical map of $S$ is birational onto its image by Lemma \ref{birational}.

The proof of Theorem \ref{6} is quite complicated. Let's briefly outline the steps of the proof.  We first show that the eventual paracanonical map of $\hat{S}$ is indeed a biraional morphism. Then we realized that it is crucial to study   $\hat{S}$ via the morphism induced by the linear system $|K_{\hat{S}}-\hat{a}^*L|$. From which, we prove that $\deg a_S=\deg \hat{a}\leq 7$.  We finally apply the symmetry induced by the Heisenberg groups and other delicate geometric considerations to finish the proof of Theorem \ref{6}.

\begin{prop}\label{eventual=canonical}Assume that $K_S^2=6$ and $\deg a_S>2$, the eventual paracanonical map of  $\hat{S}$ is a morphism.
\end{prop}
\begin{proof}We just need to show that $F_1$ and $F_2$ in (\ref{canonical-bundle}) are $0$ and equivalently $\hat{F_1}$ and $\hat{F_2}$ in (\ref{canonical-bundle1}) are zero. We know that $(K_{\hat{S}}^2)=6h^0(L)^2$, $(K_{\hat{S}}\cdot \hat{a}^*L)=4h^0(L)^2$, and $\deg \hat{a}=\deg a_S=1+h^0(L)-r$, where $r:=\mathrm{length}(Z)$. We also recall that $L^-=(-1_{\hat{A}})^*L$ is algebraically equivalent to $L$. Hence $(K_{\hat{S}}\cdot L^-)=4h^0(L)^2$ and $(L^-)^2=2h^0(L)$. Thus
\begin{eqnarray*}&&(\hat{\Phi'}^*\cO_{\mathbb P^{h-1}}(1))^2\\&=&(\hat{\rho}^*(K_{\hat{S}}-\hat{F_1})-\hat{F_2}-\hat{\rho}^*\hat{a}^*L^-)^2 \\
&=&(K_{\hat{S}}^2)-2(K_{\hat{S}}\cdot \hat{a}^*L^-)+(\hat{a}^*L^-)^2-2(K_{\hat{S}}\cdot \hat{F_1})+(\hat{F_1})^2+(\hat{F_2})^2\\
&=&2(1-r)h^0(L)-2(K_{\hat{S}}\cdot \hat{F_1})+(\hat{F_1})^2+(\hat{F_2})^2.
\end{eqnarray*}

If $F_1\neq 0$, we know that $(F_1^2)<0$ and $K_{S}\cdot F_1\geq 0$ since $F_1$ is $a_S$-exceptional.  We verify easily that $-2(K_S\cdot F_1)+(F_1^2)\leq -2$. Thus
$-2(K_{\hat{S}}\cdot \hat{F_1})+(\hat{F_1})^2\leq -2h^0(L)^2$, which implies that $(\hat{\Phi'}^*\cO_{\mathbb P^{h^0(L)-1}}(1))^2<0$, which is impossible.

If $F_2\neq 0$, the same argument implies that $r=0$, $(F_2^2)=-1$, $(\hat{F_2})^2=-h^0(L)^2$, and $h^0(L)=2$. In this case, $F_2$ is supported on the exceptional locus of $S'\rightarrow S$   and $\hat{S'}$ is birational onto a hypersurface $W$ of $\hat{A}\times \mathbb P^1$. Since $\deg \hat{a}=\deg a_S=3$ in this case, $W\in |M\boxtimes \cO_{\mathbb P^1}(3)|$, where $M$ is a line bundle on $\hat{A}$. Moreover, from the construction of $\hat{\Phi}'$, we know that $\hat{\Phi}'^*(L^-\boxtimes \cO_{\mathbb P^1}(1))=\hat{\rho}^*K_{\hat{S}}-\hat{F_2}$. Since $\hat{S}'$ is a smooth model of $W$ and $W$ is Gorenstein, $K_{\hat{S}'}$ is a subsheaf of $\hat{\Phi}'^*(M\boxtimes \cO_{\mathbb P^1}(1))$. Thus $\hat{\rho}^*\hat{a}^*(M-L^-)-\hat{F_2}$ is effective. Thus $M-L^-$ is a non-zero effective line bundle on $\hat{A}$. On the other hand, $$5h^0(L)^2=(\hat{\rho}^*K_{\hat{S}}-\hat{F_2})^2=(L^-\boxtimes \cO_{\mathbb P^1}(1))^2\cdot (M\boxtimes \cO_{\mathbb P^1}(3)).$$ We deduce that $(L^-\cdot M)=4$ and thus by Hodge index, $M\equiv L^-$ and we have a contradiction.

\end{proof}

This proposition is quite important in the classification of minimal surfaces with $p_g=q=2$ and $K^2=6$. We will  apply it to verify \cite[Assumption 0.3]{AC} in the study of AC3 surfaces.

The structure of the eventual paracanonical morphism $$\hat{\Phi}: \hat{S}\rightarrow \mathbb P_{\hat{A}}(\varphi_L^*\cM)\subset \hat{A}\times \mathbb P^{h^0(L)-1} $$ plays an important role in the classification of $S$. Let $\psi: \hat{S}\rightarrow \mathbb P^{h^0(L)-1}$ be the natural morphism. We denote by $\eta=\psi^*\cO_{\mathbb P^{h-1}}(1)$ and thus $\eta=K_{\hat{S}}-\hat{a}^*L^-$.
 Note that $(\eta^2)=2(1-r)h^0(L)\geq 0$. Thus the following lemma is clear.
 \begin{lemm} We have $0\leq r\leq 1$ and $\eta$ is big iff $r=0$ and the image of $\psi$ is a curve when $r=1$.
 \end{lemm}

We   have the following commutative diagram
\begin{eqnarray*}
\xymatrix{& &\mathbb P^{h^0(L)-1}\\
\hat{S}\ar[urr]^{\psi}\ar[rr]^{\hat{\Phi}}\ar[drr]_{\hat{a}} && \hat{A}\times \mathbb P^{h^0(L)-1}\ar[d]\ar[u]\\
&&\hat{A}.}
\end{eqnarray*}
\begin{lemm}\label{psi} The morphism $\psi$ is induced by the complete linear system $|\eta|$.
\end{lemm}
Since $\hat{a}_*\omega_{\hat{S}}=\cO_{\hat{A}}\oplus   \varphi_L^*\cM$, and $$0\rightarrow \phi_L^*(-1_A)^*\widehat{L|_Z}^{\vee}\rightarrow H^0(L)^{\vee}\otimes L^-\rightarrow \phi_L^*\cM\rightarrow 0,$$ we have $H^0(\hat{S}, \eta)=H^0(\hat{S}, K_{\hat{S}}-\hat{a}^*L^-)\simeq H^0(L)^{\vee}$.
\begin{prop}\label{K_L}
The  morphism $\psi$ is $\mathcal K_L$-equivariant and the image of $\psi$ is  $K_L$-invariant. Moreover,  the representation of $\mathcal K_L$ on $H^0(\eta)$ is isomorphic to the dual of the Schr\"odinger representation.
\end{prop}
\begin{proof} We have already seen in Lemma \ref{psi} that $\psi$ is induced by the complete linear series $|K_{\hat{S}}-\hat{a}^*L^-|$. Note that $\mathcal K_L$ acts on $\hat{a}^*L^-$   via the action of $\mathcal K_L$ on $L^-$. Since $\hat{S}\rightarrow S$ is a $K_L$-\'etale cover, we have the regular action of $K_L$ on $K_{\hat{S}}$, which induces an action of $\mathcal K_L$ on $K_{\hat{S}} $. Thus we have a linear action of $\mathcal K_L$ on $\eta$. Thus the  morphism $\psi$ is $\mathcal K_L$-equivariant and since the kernel of $\mathcal K_L\rightarrow K_L$ acts on $H^0(\eta)$ by multiplication-by-scalar, the image of $\psi$  is  $K_L$-invariant. We remark that since $\deg a_S\geq 3$, $h^0(L)\geq 2$.

Since $\hat{a}_*\omega_{\hat{S}}=\cO_{\hat{A}}\oplus \varphi_L^*\cM$,
$$H^0(K_{\hat{S}}\otimes \hat{a}^*\varphi_L^*Q)=H^0(\varphi_L^*\cM\otimes Q)\simeq H^0(\hat{S}, \eta)\otimes \hat{a}^*H^0(\hat{A}, L^-\otimes\varphi_L^*Q)$$ for $Q\in \Pic^0(A)$ general, where the second isomorphism holds because of the short exact sequence (\ref{exact1}).

There exists a unique $\mathcal K_L$-invariant section of $$H^0(\hat{S}, \eta)\otimes \hat{a}^*H^0(\hat{A}, L^-\otimes\varphi_L^*Q),$$ which descends to the unique section (up to scalar) of $$H^0(\hat{A}, \varphi_L^*\cM\otimes\varphi_L^*Q)^{K_L}=H^0(A, \cM\otimes Q)= H^0(K_S\otimes a_S^*Q).$$ Since the representation of $\mathcal K_L$  on $H^0(\hat{A}, L^-\otimes\varphi_L^*Q)$  is isomorphic to the Schr\"odinger representation, the representation of $\mathcal K_L$ on $H^0(\eta)$ is isomorphic to the dual of the Schr\"odinger representation.
\end{proof}
\subsection{The case where $r>0$}
\begin{theo}\label{AC3}
When $r=1$, we have $h^0(L)=3$, the image of $\psi$ is a genus $1$ curve $C\subset \mathbb P^2$, and $S$ belongs to the AC3 family.

\end{theo}
\begin{proof} Let $C$ be the normalization of the image of $\psi$. We then have a factorization of morphism $$\psi: \hat{S}\rightarrow C\rightarrow \mathbb P^{h^0(L)-1}$$ and $K_L$ acts faithfully on $C$ too.

 Then $\hat{S}$ is the minimal smooth model of a divisor of $\hat{A}\times C$.

Since $r=1$, $Z$ is a reduced point and we have a short exact sequence
\begin{eqnarray}\label{divisor-sequence}0\rightarrow P\rightarrow H^0(\eta)\otimes L^-\rightarrow \varphi_L^*\cM\rightarrow 0,\end{eqnarray} where $P\in \Pic^0(\hat{A})$ is a torsion line bundle.
In particular,
\begin{equation*}
q(\hat{S})=\begin{cases}
3, &\text{if $P$ is trivial},\\
2, &\text{if $P$ is non-trivial}.
\end{cases}
\end{equation*}
  Since  $\hat{a}$ is generically finite,  the Albanese morphism $a_{\hat{S}}: \hat{S}\rightarrow A_{\hat{S}}$ of $\hat{S}$ is also generically finite onto its image. Note that the morphism $\hat{S}\rightarrow C\subset JC$ factors through $a_{\hat{S}}$ and $A_{\hat{S}}\rightarrow JC$ is surjective. Thus  $g(C)\leq 2$.
Moreover, if $g(C)=2$, by the universal property of $a_{\hat{S}}$, $q(\hat{S})=3$ and we have the following commutative diagram
\begin{eqnarray*}
\xymatrix{
\hat{S}\ar[dr]\ar[r] & a_{\hat{S}}(\hat{S})\ar[d]\ar@{^{(}->}[r] &A_{\hat{S}}\ar[d]\\
& C\ar@{^{(}->}[r] & JC}.
\end{eqnarray*}
We then conclude that the image of $a_{\hat{S}}$ is an elliptic curve bundle over $C$. On the other hand, the morphism $\Phi: \hat{S}\rightarrow \hat{A}\times C\rightarrow \hat{A}\times \mathbb P^{h^0(L)-1}$ is birational onto its image. We conclude that $\hat{S}$ is birational to an elliptic bundle over $C$, which is a contradiction.

Thus $g(C)\leq 1$.

If $C=\mathbb P^1$, then $K_L$ is a subgroup of $\mathrm{Aut}(C)=\mathrm{PGL}_2$. We  recall that the automorphism group $\mathrm{PGL}_2$  has only $\mathbb Z_r$, $D_r$, $A_4$, $S_4$ and $A_5$  as finite subgroups (see \cite{B2}). Since $h^0(L)\geq 3$, $K_L$ is a non-cyclic abelian group of order $h^0(L)^2$, we then get a contradiction.


 We then assume that $C$ is of genus $1$. In this case, $C\subset \mathbb P^{h^0(L)-1}$ and $\deg_C(\cO_{\mathbb P^{h^0(L)-1} }(1))=h^0(L)\geq 3$.

Since $\hat{S}$ is birational to an ample divisor of the abelian threefold $\hat{A}\times C$, $q(\hat{S})=3$ and hence in (\ref{divisor-sequence}), $P=\cO_{\hat{A}}$ is trivial.

 Note that $C\subset \mathbb P^{h^0(L)-1}$ is $K_L$-invariant. From the  Schr\"odinger representation, we see that $K_L$ acts on $C$ faithfully and hence $K_L$ is a subgroup of $\mathrm{Aut}(C)$. Since $q(S)=2$,  $C/K_L$ is rational.

  We know that $ \mathrm{Aut}(C)$ is the semi-product of $C$ with
  \begin{equation*}
   \begin{cases}\langle \sqrt{-1}\rangle \simeq\mathbb Z_4, & \text{if $C\simeq \mathbb C/(\mathbb Z\oplus \mathbb Z\sqrt{-1})$}, \\
   \langle -e^{\frac{2\pi i}{3}}\rangle\simeq \mathbb Z_6, & \text{if $C\simeq \mathbb C/(\mathbb Z\oplus\mathbb Ze^{\frac{2\pi i}{3}})$},\\
   \langle -1\rangle \simeq \mathbb Z_2, & \text{otherwise}.
  \end{cases}
  \end{equation*}
  We then verify that the only possibility is that $K_L\simeq \mathbb Z_2\times \mathbb Z_2$ or $\mathbb Z_3\times \mathbb Z_3$. Since $h^0(L)\geq 3$, $h^0(L)=3$.

We now show that the image $W$ of $\hat{\Phi}$ is the canonical model of $\hat{S}$, which verifies \cite[Assumption 0.3]{AC} for $S$.
We have already seen that $\hat{\Phi}: \hat{S}\rightarrow W$ is a finite birational morphism and $\hat{\Phi}^*((L^-\boxtimes \cO_{\mathbb P^2}(1))|_W)=K_{\hat{S}}$, it suffices to show that $W$ is normal, Gorenstein, and $K_W=(L^-\boxtimes \cO_{\mathbb P^2}(1))|_W$.
Indeed, $W$ is contained in the two divisors $D_1:=\hat{A}\times C\subset \hat{A}\times \mathbb P^2$ and $D_2:=\mathbb P_{\hat{A}}(\varphi_L^*\cM)\subset \hat{A}\times \mathbb P^2$, which is a divisor in $|L^-\boxtimes \cO_{\mathbb P^{2}}(1)|$ corresponding to the short exact sequence (\ref{divisor-sequence}). Thus, as a divisor in the smooth abelian $3$-fold $D_1$, $W$ is Gorenstein and $K_W$ is a subsheaf of $D_2|_{W}=(L^-\boxtimes \cO_{\mathbb P^2}(1))|_W$. Moreover, since $\hat{S}$ is birational to $W$, $K_{\hat{S}}\subset\hat{\Phi}^*K_W$ and equality holds iff $W$ is normal. We then conclude that $K_W=(L^-\boxtimes \cO_{\mathbb P^2}(1))|_W$, $W$ is normal, and thus $W$ is the canonical model of $\hat{S}$.

  By \cite[Theorem 0.8]{AC} and Proposition \ref{K_L}, we conclude that $S$ is an AC3-surface.

\end{proof}
\subsection{Boundedness of $h^0(L)$}

In this subsection, we assume that $r=0$ and hence $\varphi_L^*\cM=H^0(L)^{\vee}\otimes L^-$ is locally free,
$(\eta^2)=2h^0(L)>0$. Thus $h^0(L)\geq 3$.
 Moreover, $(\hat{a}^*L)^2=2h^0(L)(h^0(L)+1)$ and $(\hat{a}^*L\cdot \eta)=2h^0(L)(h^0(L)-1)$. We also note that $q(\hat{S})=2$ and thus $\hat{a}$ is the Albanese morphism of $\hat{S}$. Let $W\subset \mathbb P^{h^0(L)-1}$ be the image of $\psi$. We now list the possible structure of $W$.
\begin{theo}\label{boundness}
Assume that $r=0$, we have $h^0(L)\leq 6.$ More precisely, there are the following possibilities  of $\psi$:
\begin{itemize}
\item $h^0(L)=3$ and $\psi: \hat{S}\rightarrow W=\mathbb P^2$ is of degree $6$;
\item $h^0(L)=4$ and $\psi:  \hat{S}\rightarrow W\subset \mathbb P^3$ is birational onto an octic;
\item $h^0(L)=4$ and $\psi:  \hat{S}\rightarrow W\subset \mathbb P^3$ is of degree $2$ over a quartic;
\item $h^0(L)=4$ and $\psi:  \hat{S}\rightarrow W\subset \mathbb P^3$ is of degree $4$ onto a quadric;
\item $h^0(L)=6$ and $\psi:  \hat{S}\rightarrow W\subset \mathbb P^5$ is of degree $3$ onto a surface of degree $4$.

\end{itemize}
\end{theo}
\begin{proof}\
Since $W$ is non-degenerate, $\deg(W)\geq h^0(L)-2$.

 If $\deg(\hat{S}/W)\geq 3$, we have $$2h^0(L)=(\eta^2)=\deg(\hat{S}/W)\deg(W)\geq 3(h^0(L)-2).$$ Thus $h^0(L)\leq 6$. The only possible cases are that $h^0(L)=6$ and $\psi$ is of degree $3$ onto its image or  $h^0(L)=4$ and $\psi$ is of degree $4$ onto its image or  $h^0(L)=3$ and $\psi$ is of degree $6$ onto its image.

If $\psi: \hat{S}\rightarrow W\subset \mathbb P^{h^0(L)-1}$  is birational onto its image, then $h^0(L)\geq 4$. Let $C$ be a general hyperplane section of $\eta$. Since $(K_{\hat{S}}\cdot \eta)=2h^0(L)^2$, $g(C)=h^0(L)^2+h^0(L)+1$.
By the proof of \cite[Lemma 5.1]{B}, we have $\frac{2h^0(L)+4}{3}\geq h^0(L)-1$ and hence $h^0(L)\leq 7$. To be more precisely, we apply Castelnuovo's analysis (see \cite[Page 116]{ACGH}). Let $m:=\lfloor \frac{2h^0(L)-1}{h^0(L)-3} \rfloor=2+\lfloor \frac{5}{h^0(L)-3}\rfloor$ and write $$2h^0(L)-1=m(h^0(L)-3)+\epsilon.$$  Then Castelnuovo's inequality implies that $$h^0(L)^2+h^0(L)+1=g(C)\leq \frac{m(m-1)}{2}(h^0(L)-3)+m\epsilon.$$ We then verify that it is only possible that $h^0(L)=4$ and in this case $C$ is isomorphic to a smooth plane octic.

 In the rest of the proof,  we assume that $\psi: \hat{S}\rightarrow W$ is of degree $2$ and then $\deg(W)=h^0(L)$. We denote by $\tau$ the corresponding involution on $\hat{S}$ and we have $$\hat{S}\xrightarrow{\iota} \hat{S}/\langle \tau\rangle\xrightarrow{\rho} W\subset \mathbb P^{h-1}.$$ If $\hat{S}/\langle \tau\rangle$ is of general type, let $D$ be a general hyperplane of $\eta':=\rho^*\cO_W(1)$. Then $D$ is a smooth projective curve and $\rho^*\cO_W(1)|_D$ is a special divisor since $K_D=(K_{\hat{S}/\langle \tau \rangle}+\eta')|_D$. Then Clifford's lemma implies that $$\deg (\eta'|_D)=h^0(L)\geq 2(h^0(L)-2)$$ and hence $h^0(L)\leq 4$.

  Thus we may assume in the following that $\hat{S}/\langle \tau\rangle$ is not of general type. We then write $\iota_*\cO_{\hat{S}}=\cO_{K_{\hat{S}/\langle \tau\rangle}}\oplus V^{-1}$ and $K_{\hat{S}}=\iota^*(K_{\hat{S}/\langle \tau\rangle}+V)$, where $V$ is a  divisor and $2V$ is linearly equivalent to the branched divisor of $\iota$. We conclude that $\hat{a}^*L^-\sim \iota^*(K_{\hat{S}/\langle \tau\rangle}+V-\eta')$.
  Moreover, since $K_{\hat{S}/\langle \tau\rangle}-\eta'$ has no global sections,
   we see that the morphism on $\hat{S}$ induced by the complete linear system $|\hat{a}^*L^-|$ also factors birationally through $\psi$.

We can also study $|\hat{a}^*L^-|$ via the morphism $\hat{a}$. Since
 $$\hat{a}_*\cO_{\hat{S}}=\cO_{\hat{A}}\oplus (H^0(L)\otimes (L^-)^{-1}),$$ we see that the linear system $|\hat{a}^*L^-|$ separates general fibers of $\hat{a}$. By Reider's theorem (see \cite[Section 10.4]{BL}), we know that if $h^0(L)=h^0(L^-)\geq 5$ and there exists no elliptic curve $E$ on $\hat{A}$ such that $(L^-\cdot E)\leq 2$, $L$  is a very ample divisor on $\hat{A}$, which implies that $|\hat{a}^*L^-|$ induces a birational morphism of $\hat{S}$. Thus we have $h^0(L)\leq 4$ or there exists an elliptic curve $E$ on $\hat{A}$ such that $(L^-\cdot E)\leq 2$.

  We now argue that we still have $h^0(L)\leq 4$ in the latter case.  Assume the contrary that $h^0(L)\geq 5$. Then there exists a unique elliptic curve $E$ such that $d:=(L^-\cdot E)\leq 2$ by Poincar\'e's reducibility. Let $\hat{A}\rightarrow \hat{A}/E$ be the quotient and we have a commutative diagram
 \begin{eqnarray*}
 \xymatrix{
 \hat{S}\ar[r]^{\hat{a}}\ar[dr]_q & \hat{A}\ar[d]\\
 & \hat{A}/E.}
 \end{eqnarray*}
 It is easy to verify that $q$ is also a fibration since $\hat{a}$ is the Albanese morphism of $\hat{S}$ and we denote by $C$ a general fiber of $q$.  We consider the morphism $\hat{a}|_C: C\rightarrow E$. By base change, we have $$(\hat{a}|_C)_*\omega_C=\cO_E\oplus (L^-|_E)^{\oplus h^0(L)}.$$ Thus $g(C)=1+dh^0(L)$. On the other hand, $\deg (\hat{a}^*L^-|_{C})=d(h^0(L)+1)$ and $h^0(C, \hat{a}^*L^-|_C)=h^0(L)+d$. It is clear that $\hat{a}^*L^-|_C$ is special. Thus by Clifford's lemma, $d=2$. We apply again Reider's theorem to conclude that $|L^-|$ induces a morphism of  $\hat{A}$ of degree $2$ and we denote by $\tau_A$ the corresponding involution on $\hat{A}$. The quotient $\hat{A}/\langle \tau_A\rangle$ is a $\mathbb P^1$-bundle over $\hat{A}/E$. We then have the commutative diagram
 \begin{eqnarray}\label{involution}
 \xymatrix{
 \hat{S}\ar@/^2pc/[rr]^{\psi}\ar[r]\ar[d]^{\hat{a}} & \hat{S}/\langle \tau\rangle\ar[r]^{\rho}\ar[d] & W\ar@{^{(}->}[r] & \mathbb P^{h^0(L)-1}\\
 \hat{A}\ar[r] & \hat{A}/\langle \tau_A \rangle}
 \end{eqnarray}

 We still take $D$   a general section of $\eta'$. Consider the short exact sequence
 $$0\rightarrow \cO_{ \hat{S}/\langle \tau\rangle} \rightarrow \eta'\rightarrow \eta'|_D\rightarrow 0.$$
 Note that
 \begin{eqnarray}\label{h1}H^1(\hat{S}/\langle \tau\rangle, \eta')&=&H^1(\hat{S}, \eta)^{\tau}=H^1(\hat{S}, K_{\hat{S}}-\hat{a}^*L^-)^{\tau}\\\nonumber
 &\simeq& (H^1(\hat{S}, \hat{a}^*L^-)^{\vee})^{\tau}=(H^1(\hat{A}, L^- \oplus (\cO_A)^{\oplus h^0(L)})^{\vee})^{\tau_A}.
 \end{eqnarray}
 Thus $h^1(\hat{S}/\langle \tau\rangle, \eta')=h^0(L)q(\hat{A}/\langle \tau_A \rangle)=h^0(L)$. We thus conclude that $h^1(\eta'|_D)>0$ and hence $\eta
'|_D$ is again special.
We conclude again by Clifford's lemma that  $$h^0(L)=\deg(\eta'|_D)\geq 2(h^0(\eta'|_D)-1)\geq 2(h^0(L)-2)$$ and hence $h^0(L)\leq 4$.

\end{proof}
\subsection{The case that $h^0(L)=6$ cannot occur}

We  exclude the case that $h^0(L)=6$ via the symmetry induced by the finite Heisenberg group.

\begin{lemm}\label{symmetry}The case that $h^0(L)=6$ cannot occur.
\end{lemm}
\begin{proof}
In this case $W\subset \mathbb P^5$ is a surface of minimal degree. Thus we know that $W$ is smooth and $(W, \cO_{\mathbb P^5}(1)|_W)$ is isomorphic to one of the followings (see \cite[Theorem 1]{EH}):
\begin{itemize}
\item a cone over a rational quartic curve;
\item $(\mathbb P_{\mathbb P^1}(\cO_{\mathbb P^1}(1)\oplus \cO_{\mathbb P^1}(3)), H)$, where $H$ is the relative hyperplane divisor;
\item $(\mathbb P^1\times \mathbb P^1, \cO_{\mathbb P^1}(1)\boxtimes  \cO_{\mathbb P^1}(2))$;
\item $(\mathbb P^2, \cO_{\mathbb P^2}(2))$.
\end{itemize}

We shall apply the $\mathcal K_L$-equivariant structure to rule out all possibilities.

Note that $W$ admits an $(\mathbb Z_6\times \mathbb Z_6)$-automorphism and let  $\sigma$ and $\tau$ be two generators of $\mathcal K_L$ such that $\sigma$ is the cyclic  permutation $$z_0\rightarrow z_1\rightarrow z_2\cdots z_4\rightarrow z_5.$$ and $\tau$ induces $$z_i\rightarrow \zeta^i z_i,\; 0\leq i\leq 5,$$
Where $\zeta$ is a primitive 6-th root of unity.

We see that $K_L$ is an automorphism group of the pair $W\subset \mathbb P^5$  and $K_L$ has no fixed points.

We also recall that the automorphism group $\mathrm{PGL}_2$ of $\mathbb P^1$ has only $\mathbb Z_r$, $D_r$, $A_4$, $S_4$ and $A_5$  as finite subgroups and there is only one conjugacy class for each finite subgroup (see \cite{B2}).

In the first case, the only singular point of $W$ is preserved by $K_L$, which would imply that the Schr\"odinger representation is reducible. This is absurd.

In the second case, $W$ contains a unique line $D$ with self-intersection $-2$, which corresponding to the quotient $\cO_{\mathbb P^1}(1)\oplus \cO_{\mathbb P^1}(3)\rightarrow \cO_{\mathbb P^1}(1)$. Thus $K_L$ preserves $D$ and the ideal sheaf of $D$ is preserved by $\mathcal K_L$, which again contradicts the irreduciblity of the Schr\"odinger representation.

In the third case, since $K_L$ moves  lines to lines, it preserve the second projection $\mathbb P^1\times \mathbb P^1\rightarrow \mathbb P^1$. We now consider the action of $K_L$ on the base, which induces a homomorphism $\gamma: K_L\rightarrow \mathrm{PGL}_2$. If the image is a cyclic group, this cyclic group has a fixed point and hence $K_L$ preserves a line of $\mathbb P^5$, which is impossible as we have seen above. Thus the image of $\gamma$ has to be $D_2$ according to the above classification of subgroups of $\mathrm{PGL_2}$. But this implies that $\mathrm{PGL}_2$ contains $\mathbb Z_3\times \mathbb Z_3$, which is again impossible.

In the last case, $K_L$ preserve $W$ and hence we have an injective homomorphism $\gamma: K_L\rightarrow \mathrm{PGL}_3=\mathrm{Aut}(\mathbb P^2)$.  By the classification of finite subgroups of $ \mathrm{PGL}_3$ (see  \cite[Corollary 4.6, Theorem 4.7, Theorem 4.8]{Dol}), the image of $\gamma$ is intransitive and hence is  conjugate to groups of automorphisms generated by diagonal matrices
\begin{eqnarray*}
&&(x_0:x_1:x_2)\rightarrow (\zeta^{a_1} x_0: \zeta^{a_2}x_1:x_2)\\
&&(x_0:x_1:x_2)\rightarrow (\zeta^{b_1}x_0: \zeta^{b_2} x_1:x_2),
\end{eqnarray*}
where $a_1$, $a_2$, $b_1$ and $b_2$ are integers.
However we can then verify easily that that the representation $\mathcal K_L$ on $H^0(L)$ is not compatible with the action of $K_L$ on $ | \cO_{\mathbb P^2}(2)|$.
\end{proof}
\subsection{Remarks on the case that $h^0(L)=3$}
When $h^0(L)=3$, $\deg a_S=4$. We consider the Stein factorization $a_S: S\rightarrow \overline{S}\xrightarrow{\rho} A$. Since $a_{S*}\omega_S$ is locally free, $a_{S*}\cO_S=\rho_*\cO_{\overline{S}}=(a_{S*}\omega_S)^{\vee}$ is locally free and $R^1a_{S*}\cO_S=0$. Thus $\overline{S}$ has rational singularities and $\rho$ is a finite flat quadruple cover.
By the main result of \cite{HM}, we know that $\rho$ is determined by a totally decomposable section of $H^0(A, \wedge^2S^2\cM\otimes \det \cM^{-1})$. These totally decomposable sections are determined by Penegini and Polizzi in \cite[Proposition 2.3 and 2.4]{PP3} as $K_L$-invariant totally decomposable sections of $H^0(\hat{A}, \varphi_L^*(\wedge^2S^2\cM\otimes \det \cM^{-1}))\simeq H^0(\hat{A}, L)^{\oplus 15}$. To be more precise, $$\dim H^0(\hat{A}, \varphi_L^*(\wedge^2S^2\cM\otimes \det \cM^{-1}))^{K_L}=5$$ and there are a smooth conic $C_{A, L}$ inside $$|H^0(\hat{A}, \varphi_L^*(\wedge^2S^2\cM\otimes \det \cM^{-1}))^{K_L}|$$ parametrizing totally decomposable sections.

Thus a minimal surface with $p_g=q=2$ and $K_S^2=6$ corresponds to a point in $C_{A, L}$ such that the corresponding quadruple cover $\overline{S}$ has suitable singularities.

 When $A$ is general in the moduli spaces of $(1, 3)$-polarized abelian surfaces, Penegini and Polizzi determined the points in $C(A, L)$ such that $a_S$ is finite (see \cite[Proposition 2.6]{PP3}).
\subsection{Remarks on the case that $h^0(L)=4$}
In this case there are two possibilities. The polarization can be of type $(1, 4)$ or  $(2,2)$. The two different polarizations have different symmetries and need to be treated differently.

\begin{lemm}The case that $W\subset \mathbb P^3$ is an octic cannot occur.
\end{lemm}
\begin{proof}
Assume the contrary that $W\subset \mathbb P^3$ is an octic.
In this case, we have seen in the proof of Theorem \ref{boundness} that a general hyperplane section of $W$ is a smooth plane octic. Thus $W$ is smooth in codimension $1$ and thus as a divisor in $\mathbb P^3$, $W$ is normal. Hence $\psi_*\cO_{\hat{S}}=\cO_W$. Since $h^1(\cO_{\hat{S}})=2$ and $h^1(\cO_W)=0$, $R^1\psi_*\cO_{\hat{S}}$ is supported on singular points, where $W$ does not have rational singularities.  Since the length of $R^1\psi_*\cO_{\hat{S}}$ is $2$, the action of $K_L$ on $\mathbb P^3$ has a fixed point or a fixed line, which is a contradiction.

\end{proof}
\begin{lemm} $W$ cannot be a quadric in $\mathbb P^3$.
\end{lemm}
\begin{proof} It is easy to verify that there exists no $K_L$-invariant quadrics when $L$ is of $(1, 4)$ type. However, if $L$ is of type $(2, 2)$, we get no contradictions from the symmetry of $K_L$.  In this case, $K_L\simeq (\mathbb Z_2)^{\oplus 4}$ and we may assume that $W\simeq \mathbb P^1\times \mathbb P^1$ is a smooth quadric. By the classification of finite subgroups of $\mathrm{Aut}(\mathbb P^1\times \mathbb P^1)$ (see \cite[Theorem 4.9]{Dol}), after a conjugation, we may assume that $K_L\simeq K_1\times K_2$, where $K_i\simeq \mathbb Z_2\times \mathbb Z_2$ acting faithfully on the $i$-th factor of $W=\mathbb P^1\times \mathbb P^1$
and acting trivially on the $(3-i)$-th factor for $i=1, 2$.  Let $\eta_i=p_i^*\cO_{\mathbb P^1}(1)$, where $p_i: \hat{S}\rightarrow \mathbb P^1$ is the natural morphism onto the $i$-th factor of $W$. It is easy to verify that $p_i$ is a fibration. Note that $\eta_1+\eta_2=\eta$ and recall that $$(K_{\hat{S}}\cdot (\eta_1+\eta_2))=(K_{\hat{S}}^2)-(K_{\hat{S}}\cdot \hat{a}^*L^-)=32.$$

Let $S_i:=\hat{S}/K_i$ and $B_i:=\hat{A}/K_i$ for $i=1, 2$. Note that $p_{3-i}$ is preserved by $K_i$ and we thus have a factorization $p_{3-i}: \hat{S}\rightarrow S_i\xrightarrow{q_i} \mathbb P^1$ and the following commutative diagram:
\begin{eqnarray*}
\xymatrix{
\mathbb P^1\ar@{=}[d] &\hat{S}\ar[l]^{p_{3-i}}\ar[r]\ar[d]^{/K_i} & \hat{A}\ar[d]^{/K_i}\\
\mathbb P^1 & S_i\ar[r] \ar[l]^{q_i}& B_i.}
\end{eqnarray*}
We see that a general fiber of $p_{3-i}$ is an \'etale $K_i$-cover of the corresponding fiber $C_i$ of $q_i$. In particular, $(K_{\hat{S}}\cdot \eta_{3-i})=4(2g(C_i)-2)$  for $i=1, 2$. Note that $g(C_i)>2$, since otherwise $q_i$ is birationally trivial (see \cite[Lemme in the appendix]{D}), which is obviously impossible. We thus conclude that $g(C_1)=g(C_2)=3$.

Note that $S_i/K_{3-i}=S$ and the action of $K_{3-i}$  on $\mathbb P^1$ is conjugate to the subgroup generated by two involutions $$(x: y)\rightarrow (y: x), \; (x: y)\rightarrow (x: -y).$$ We have the following commutative diagram:
\begin{eqnarray*}
\xymatrix{
\mathbb P^1\ar[d]_{/K_{3-i}} &S_i\ar[l]^{q_{i}}\ar[r]\ar[d]^{/K_{3-i}} &B_i\ar[d]\\
\mathbb P^1 & S\ar[r] \ar[l]^{t_i}& A.}
\end{eqnarray*}
Thus $S$ admits two fibrations $t_1$ and $t_2$ onto $\mathbb P^1$, whose general fibers still denoted by $C_i$ are genus $3$ curves. Let $f_1\in H^0(K_S)$ be a  section by pulling-back the canonical divisor of $A$ and let $f_2\in H^0(K_S)$ be a section coming from $H^0(A, \cM)$. By \cite[the proof of Theorem 1]{X2}, the image of $$H^0(S, K_S)\rightarrow H^0(C_i, K_{C_i})$$ is of dimension $1$. Thus,
 for $t\in \mathbb C$ general, the divisor of $f_1-tf_2$ contains a general fiber of $t_1$ and $t_2$. This implies that $K_{\hat{S}}-4\eta$ is effective. This is a contradiction, since $h^0(K_{\hat{S}})=17$ and $h^0(\mathbb P^1\times \mathbb P^1, \cO_{\mathbb P^1}(4)\boxtimes\cO_{\mathbb P^1}(4) )=25$.
 \end{proof}

\begin{lemm}The case that $W$ is a quartic in $\mathbb P^3$ cannot occur.

\end{lemm}
\begin{proof}We apply the same argument in the proof of Theorem \ref{boundness}.

When $W$ is a quartic, $\psi: \hat{S}\rightarrow W$ is of degree $2$. Let $\tau$ and $\tau_A$ be respectively the involution on $\hat{S}$ and $\hat{A}$.
In this case, the morphism induced by $|\hat{a}^*L|$ also factors birationally through $\psi$. The  linear system $|L|$ also factors through the quotient of $\tau_A$. We recall the commutative diagram (\ref{involution}):
\begin{eqnarray*}
 \xymatrix{
 \hat{S}\ar@/^2pc/[rr]^{\psi}\ar[r]\ar[d]^{\hat{a}} & \hat{S}/\langle \tau\rangle\ar[r]^{\rho}\ar[d]^{a_{\tau}} & W\ar@{^{(}->}[r] & \mathbb P^{h^0(L)-1}\\
 \hat{A}\ar[r] & \hat{A}/\langle \tau_A \rangle.}
 \end{eqnarray*}

 If $L$ is a $(1, 4)$-polarization, we apply the results in \cite{BLvS}.
 By \cite[Section 2]{BLvS}, we know that
 the morphism induced by $|L|$ is not birational onto its image iff  there exists an elliptic curve $E$ on $\hat{A}$ such that $(E\cdot L)=2$ and in this case the morphism induced by $|L|$ is  birationally equivalent to the quotient of $\hat{A}$ by $\tau_A$, where  $ \hat{A}/\langle \tau_A \rangle$ is a $\mathbb P^1$-bundle over $\hat{A}/E$. Thus $q(\hat{S}/\tau)=1$. We now apply the same argument below (\ref{involution}). Let $\eta'=\rho^*\cO_W(1)$ and $D$ a general section of $\eta'$. By the same computation in (\ref{h1}), we see that $h^1(\eta'|_D)\geq h^1(\eta')-1\geq 3$. On the other hand, since $\rho$ is birational,  $D$ is birational onto a quartic in $\mathbb P^2$. Thus $\eta'|_D-K_D$ is effective, which is a contradiction.


 If $L=2\Theta$ is a $(2, 2)$-polarization and $\Theta$ is an irreducible principal polarization, $|L|$ defines an embedding of  $\hat{A}/\langle (-1)_{\hat{A}}\rangle$ in $\mathbb P^3$. Thus $\tau_A=(-1)_{\hat{A}}$.
Note that the canonical bundle of $\hat{A}/\langle \tau_A \rangle$ is trivial and $K_{\hat{S}/\langle \tau\rangle}$ is a subsheaf of $\rho^*K_W=\cO_{\hat{S}/\langle \tau\rangle}$. Thus $K_{\hat{S}/\langle \tau\rangle}$ is also trivial. Thus $a_{\tau}$ is quasi-\'etale, which implies that the quotient $\hat{a}$ is also quasi-\'etale, which is impossible. We then assume that $(\hat{A}, \Theta)\simeq (E_1, \Theta_1)\times (E_2, \Theta_2)$ is a split PPAV.  Then $\tau_A$ could be $(-1)_{\hat{A}}$, or $(-1_{E_1})\times \mathrm{Id}_{E_2}$, or $\mathrm{Id}_{E_1}\times (-1_{E_2})$.
In the first case, we apply the same argument when $\Theta$ is irreducible to get a contradiction.
In the last two cases, we can argue exactly as the $(1, 4)$-polarization case to get a contradiction.

\end{proof}

\end{document}